\documentclass[a4paper,11pt,reqno,oneside,abstracton]{amsart}

\usepackage{scalerel,stackengine,relsize}
\stackMath
\usepackage{amssymb}
\usepackage{bbm}
\usepackage{ifthen}
\usepackage{tikz}
\usepackage{tikz-3dplot}
\usepackage{bm}
\usepackage[hyphens]{url} 
\usepackage{amsmath,amssymb,amsfonts,amsthm,amsopn,dsfont,mathrsfs}
\usepackage{esint}
\usepackage{graphicx}
\usepackage{enumitem}
\usepackage{bigints}
\usepackage{fancyhdr}
\usepackage{hyperref}
\hypersetup{
    colorlinks=true,
    linkcolor=blue,
    filecolor=magenta,      
    urlcolor=cyan,
    }
\usepackage[T1]{fontenc}
\usepackage{mathtools}

\DeclarePairedDelimiter\floor{\lfloor}{\rfloor}
\usepackage[utf8]{inputenc}
\usepackage{verbatim} 
\usepackage[backend=bibtex,style=numeric,bibencoding=utf8,language=auto,autolang=other,maxnames=10]{biblatex}
\usepackage[a4paper, total={6in, 9in}]{geometry}

\newcommand*\quot[2]{{^{\textstyle #1}\big/_{\textstyle #2}}}
\newcommand\blfootnote[1]{%
  \begingroup
  \renewcommand\thefootnote{}\footnote{#1}%
  \addtocounter{footnote}{-1}%
  \endgroup
}
\newtheorem{thm}{Theorem}[section]

\newtheorem{lem}[thm]{Lemma}

\title{Hamiltonian stationary maps with infinitely many singularities}

\author{Filippo Gaia}
\date{\today}

\begin{document}

\begin{abstract}
For any $k\in \mathbb{N}$ we construct an Hamiltonian stationary Lagrangian map from a disc to $\mathbb{C}^2$ with infinitely many Schoen-Wolfson singularities which is of class $C^k$ up to the boundary and has smooth trace.
\end{abstract}
\maketitle
\blfootnote{\textup{2020} \textit{Mathematics Subject Classification}: \textup{53D12, 53C42}}
\section{Introduction}
In this notes we give a positive answer to the following question, raised in \cite{GOR}:

\vspace{2mm}
    \textit{Is it possible to construct Hamiltonian stationary Lagrangian discs with locally finite area but infinitely many Schoen-Wolfson singularities accumulating at the boundary?}
    
\vspace{2mm}
More precisely we will prove the following result:
\begin{thm}\label{thm: main theorem}
    For any $p<2$, $j\in \mathbb{N}\smallsetminus\{0\}$ there exists a weakly conformal, Hamiltonian stationary, Lagrangian map $\Phi\in C^{\sqrt{j^2+j}}( D^2, {C}^2)$ satisfying the following properties:
    \begin{enumerate}
        \item $\Phi$ has infinitely many, isolated Schoen-Wolfson singularities,
        \item $\Phi$ is continuous up to the boundary and its trace is smooth,
        \item The Lagrangian angle $g$ of $\Phi$ lies in $W^{1,p}\cap W^{1,(2,\infty)}_{\text{loc}}(D^2)$.
    \end{enumerate}
    Moreover $\Phi$ is a smooth branched immersion away from the singularities.
\end{thm}
For the precise meaning of the concepts appearing in Theorem \ref{thm: main theorem} we refer to Section 2.\\
In \cite{SW}, R. Schoen and J. Wolfson showed that in a compact symplectic $4$-manifold, the Lagrangian homology is generated
by classes that can be represented by Lagrangian Lipschitz maps which
are branched immersions except at finite number of singular points, having Schoen-Wolfson cones as tangent maps.
They also provide examples of Lagrangian homology classes where the minimizer for the area 
must have singularities of this kind (see also \cite{Wolfson}). Moreover the authors say that they expect the regularity result to hold for Lagrangian stationary maps as well.\\
In \cite{rivière2023area} and \cite{pigati2024parametrized}, A. Pigati and T. Rivière studied critical points of the area functional among lagrangian surfaces by means of parametrized integer legendrian varifolds, and conjectured that any such critical point should be realised by a smooth branched immersion away from isolated Schoen-Wolfson conical singularities.\\
Theorem \ref{thm: main theorem}, in contrast, shows that an hamiltonian stationary map $\Phi$ from a disc to $\mathbb{C}^2$ which is a branched immersion away from isolated singularities might have accumulation of singular points at the boundary, even if its trace is smooth and even if $\Phi$ is assumed to be in $C^k$ for some $k\in \mathbb{N}$. It remains open if there exist a weakly conformal map $\Phi$ minimizing the area (or at least stable) among lagrangian maps with a given trace and satisfying the properties listed in Theorem \ref{thm: main theorem}.\\
 In \cite{GOR}, G. Orriols, T. Rivi\`ere and the author introduced a variational method to construct Hamiltonian stationary Lagrangian surfaces with prescribed singularities, by which several examples of Hamiltonian stationary discs in $\mathbb{C}^2$ with finitely many singularities can be constructed. It seems difficult, however, to apply directly the method of \cite{GOR} to construct examples of Hamiltonian stationary Lagrangian discs with infinitely many singularities in $W^{1,2}(D^2)$.
Therefore the question we reported at the beginning of the Introduction remained open.\\
The construction we propose here is based on the following observation from \cite{GOR}: let $g$ be a map in $W^{1,1}(D^2,\mathbb{S}^1)$ and let $u\in W^{1,2}(D^2,\mathbb{C}^2)$ be a solution of
\begin{align}\label{eq: div-g-nabla-u}
    \operatorname{div}(g\nabla u)=0; 
\end{align}
then there exists $v\in W^{1,2}(D^2,\mathbb{C})$ such that $\nabla^\perp v=g\nabla u$ (here $\nabla^\perp:=(-\partial_{x_2}, \partial_{x_1})^T$ ) and
\begin{align}\label{eq: def-Phi}
    \Phi=\begin{pmatrix}
    u\\-\overline{v}
    \end{pmatrix}
\end{align}
is a weakly conformal, Lagrangian map with Lagrangian angle $\overline{g}$. In particular, if $g$ satisfies
\begin{align}
    \operatorname{div}(i\overline{g}\nabla g)=0,
\end{align}
then $\Phi$ is Hamiltonian stationary (for the details of this argument see the proof of Theorem \ref{thm: main theorem} in Section 3).\\
As observed in \cite{GOR} (Section 3), this method allows to construct weakly conformal Hamiltonian stationary Lagrangian maps in $W^{1,p}(D^2)$ for $p\in [1,2)$ which have infinitely many singularities, and even examples which are everywhere discontinuous: let $g$ be an everywhere discontinuous $\mathbb{S}^1$--harmonic map in $W^{1,p}(D^2)$, as constructed by L. Almeida in \cite{Almeida}. Then $u=\overline{g}$ solves \eqref{eq: div-g-nabla-u}, and by the discussion above the map $\Phi$ defined in \eqref{eq: def-Phi} satisfies the desired properties.
The map $\Phi$, however, does not lie in $W^{1,2}(D^2)$ and the surface described by $\Phi$ has infinite area ($\Phi$ is an infinite covering of the cylinder $\mathbb{S}^1\times 0\times \mathbb{R}$).\\
Geometrically, our construction will produce a branched immersion of a disc in $\mathbb{C}^2$ passing infinitely many times through the origin--where it will have its singularities-- and smooth outside of the origin. Its trace will not be an immersed smooth curve. It remains open if it is possible to construct a map as in Theorem \ref{thm: main theorem}, whose trace parametrizes an immersed smooth curve (this would be interesting for the study of the Plateau problem in the Lagrangian setting).\\
To obtain Theorem \ref{thm: main theorem} we will combine the method described above with the following analytic result.
\begin{thm}\label{thm: solving the equation}
    For any $p<2$, $j\in \mathbb{N}\smallsetminus\{0\}$ there exists a map $g\in W^{1,p}\cap W_\text{loc}^{1,(2,\infty)}(D^2,\mathbb{S}^1)\cap C^0(\overline{D^2})$ satisfying
    \begin{align}\label{eq: thm2-eqs-g}
        \operatorname{div}(i\overline{g}\nabla g)=0\quad \text{and}\quad\operatorname{div}(i\overline{g}\nabla^\perp g)=2\pi\sum_{i=1}^\infty \delta_{p_k}\text{ in }D^2
    \end{align}
    for isolated points $p_k\in D^2$, and a map $u\in C^{\sqrt{j^2+j}}(D^2,\mathbb{C})$ satisfying
    \begin{align}\label{eq: div-eq-thm}
        \operatorname{div}(g\nabla u)=0\text{ in }D^2.
    \end{align}
    Moreover, $u$ is smooth away from the points $p_i$, $\nabla u$ vanishes at isolated points and the trace of $u$ on $\partial D^2$ is smooth.
\end{thm}
We remark that the statement of Theorem \ref{thm: solving the equation} is optimal in the following sense: if $g$ is an $\mathbb{S}^1$-valued map with $\nabla g\in L^{2,q}(D^2,\mathbb{S}^1)$ for some $q<\infty$ and satisfies $\operatorname{div}(i\overline{g}\nabla g)=0$, then $g$ has no singularities (see for instance Appendix $4$ in \cite{Almeida}). The case $\nabla g\in L^{2,\infty}(D^2)$ remains open (in \cite{Almeida}, L. Almeida conjectures that if $\nabla g\in L^{2,\infty}(D^2)$, then $g$ has only finitely many singularities).\\
To prove Theorem \ref{thm: solving the equation} we will use the following observation from \cite{GOR} (Lemma V.2): if $g\in W^{1,1}$ takes values in $\mathbb{S}^1$ and $G$ is a function such that
\begin{align}
    -i\overline{g}\nabla g=\nabla^\perp G,
\end{align}
then for any $j\in \mathbb{Z}$
\begin{align}
    g^je^{-\sqrt{j^2+j}\lvert G\rvert}
\end{align}
is a solution of \eqref{eq: div-eq-thm}.
In our construction, $g$ and $G$ will be respectively the angular part and the logarithm of the modulus of an holomorphic function $\phi$ on $D^2$, which will be chosen to be smooth up to the boundary in $D^2$ and to have infinitely many zeros (which will correspond to the singularities of $u$ and $g$).\\
More generally, this method allows to construct Hamiltonian stationary Lagrangian maps starting from any holomorphic function on $D^2$.\\
Finally we remark that our method do not allow to construct maps as in Theorem \ref{thm: main theorem} satisfying in addition
\begin{align}\label{eq: Neumann-bdy-cond}
   i\overline{g}\partial_\nu g=0\text{ on }\partial D^2,
\end{align}
as we will see in Section 4 (where we will also clarify the meaning of this condition when $g\in W^{1,p}(D^2)$). The existence of such maps would be of interest because \eqref{eq: Neumann-bdy-cond} is the boundary condition satisfied by Lagrangian free-boundary surfaces when the constraint manifold is a complex surface (see \cite{Schoen-SLS}).




\textbf{Acknowledgements.}
    I'm very grateful to Tristan Rivi\`ere for his constant guidance and support. I'm also very grateful to Gerard Orriols and Federico Franceschini for the stimulating and instructive discussions.

\section{Preliminaries}
In the present document, we will denote by $D^2\subset\mathbb{R}^2$ the open ball of radius $1$ centered at the origin.\\
Let $\omega=dx_1\wedge dy_1+dx_2\wedge dy_2$ be the standard symplectic form on $\mathbb{C}^2$. A map $\Phi\in W^{1,1}(D^2,\mathbb{C}^2)$ is \textit{Lagrangian} if $\Phi^\ast\omega=0$ a.e..\\
A map $\Phi\in W^{1,1}(D^2,\mathbb{C}^2)$ is said to be \textit{weakly conformal} if $\langle \partial_x\Phi, \partial_y\Phi\rangle=0$ and $\lvert\partial_x\Phi\rvert^2=\lvert\partial_y\Phi\rvert^2$ a.e..\\
If $\Phi\in W^{1,2}(D^2,\mathbb{C}^2)$ is Lagrangian and weakly conformal, there exists a measurable map $g: D^2\to\mathbb{S}^1$ such that
\begin{align}\label{eq: Lagrangian angle}
    \Phi^\ast(dz_1\wedge dz_2)=e^{2\lambda}\overline{g}\,dx\wedge dy \text{ a.e.},
\end{align}
where $e^{2\lambda}=\lvert\partial_x \Phi\rvert^2=\lvert\partial_y \Phi\rvert^2$
(see for instance p. 3 in \cite{SWsurvey}). The map $\overline{g}$ is called the \textit{Lagrangian angle} of $\Phi$.\\
If $\Phi\in C^1(D^2,\mathbb{C}^2)$ is a weakly conformal Lagrangian map such that $\nabla \Phi$ vanishes at isolated points, we say that $p\in D^2$ is a \textit{branch point} of $\Phi$ if $\nabla \Phi(p)=0$ and the Lagrangian angle $\overline{g}$ of $\Phi$ is continuous at $p$. If $\Phi$ is an immersion away from isolated points in $D^2$ which are branch points of $\Phi$, we say that $\Phi$ is a \textit{branched immersion}.\\
It will be convenient to identify $\mathbb{C}^2$ with the algebra of quaternions $\mathbb{H}$ with basis elements $1,I,J,K$. In this identification, $I$ corresponds to the complex number $i$.
For any weakly conformal, Lagrangian $\Phi\in W^{1,2}(D^2,\mathbb{C}^2)$ there holds
$\partial_\theta \Phi\in \operatorname{span}\{J\partial_r \Phi, K\partial_r \Phi\}$ a.e.. Thus at a.e. $x\in D^2$ there exist $a,b\in \mathbb{R}$ such that
\begin{align}
    \frac{1}{r}\partial_\theta\Phi=(aJ\partial_r \Phi+bK\partial_r \Phi).
\end{align}
Now \eqref{eq: Lagrangian angle} implies
\begin{align}
    e^{2\lambda}\overline{g}=&\partial_r \Phi_1(a(J\partial_r\Phi)_2+b(K\partial_r\Phi)_2)- \partial_r \Phi_2(a(J\partial_r\Phi)_1+b(K\partial_r\Phi)_1)\\
    \nonumber
    =&-(a+ib)\lvert\partial_r \Phi\rvert^2,
\end{align}
where the indices $1$ and $2$ refer to the complex coordinates in $\mathbb{C}^2$.
Therefore
\begin{align}\label{eq: Lag-angle-def}
    \frac{1}{r}\partial_\theta \Phi=-\overline{g}J\partial_r \Phi \text{ a.e.,}
\end{align}
so that
\begin{align}\label{eq: structural-equation}
    \operatorname{div}(g\nabla \Phi)=0.
\end{align}
Let $\Phi\in C^\infty(\overline{D}^2,\mathbb{C}^2)$ be a smooth Lagrangian embedding with Lagrangian angle $g$; $\Phi$ is said to be \textit{Hamiltonian stationary} if for any smooth function $f: \mathbb{C}^2\to \mathbb{R}$ supported away from $\Phi(\partial D^2)$ there holds
\begin{align}
    \int_{D^2}\langle d\Phi; d(I(\nabla f)\circ\Phi)\rangle=0.
\end{align}
By \eqref{eq: structural-equation} this implies
    \begin{align}\label{eq: computation-Lemma-I}
    \nonumber
        0=&\int_{D^2}\langle d\Phi; d(I(\nabla f)\circ\Phi)\rangle=-\int_{D^2}\langle\Delta \Phi, I(\nabla f)\circ\Phi\rangle
        =-\int_{D^2}\langle i\overline{g}dg\cdot d\Phi, \nabla f\circ\Phi\rangle\\
        =&-\int_{D^2}\langle i\overline{g} dg; d(f\circ \Phi)\rangle
        =\int_{D^2}\operatorname{div}(i\overline{g}\nabla g) f\circ \Phi.
    \end{align}
As $\Phi$ is an embedding and \eqref{eq: computation-Lemma-I} holds for any smooth $f$ supported away from $\Phi(\partial D^2)$, we have $\operatorname{div}(i\overline{g}\nabla g)=0$ in $D^2$.\\
More generally, for a Lagrangian map $\Phi\in W^{1,1}(D^2,\mathbb{C}^2)$ we say that $\Phi$ is \textit{Hamiltonian stationary} if its Lagrangian angle $\overline{g}$ satisfies
\begin{align}
    \operatorname{div}(i\overline{g}\nabla g)=0.
\end{align}
In \cite{SW}, R. Schoen and J. Wolfson classified all the two-dimensional Hamiltonian stationary Lagrangian cones:
for any relatively prime positive integers $p$ and $q$, up to unitary rotation, the \textit{Schoen-Wolfson cone $\Sigma_{p,q}$} 
has the following weakly conformal parametrization
\begin{align}\label{eq: Phi-p-q}
\Phi_{p,q}:\overline{D^2}\to \mathbb{C}^2,\quad r e^{i\theta}\mapsto \frac{r^{\sqrt{pq}}}{\sqrt{p+q}} \begin{pmatrix}\sqrt{q}e^{ip\theta}
\\ i\sqrt{p}e^{-iq\theta}\end{pmatrix}
\end{align}
and its Lagrangian angle is given by $e^{i(p-q)}$. Notice that $\Phi_{p,q}\in C^{\sqrt{pq}}(D^2)$, meaning that $\Phi_{p,q}\in C^{\floor{\sqrt{pq}}, \sqrt{pq}-\floor{\sqrt{pq}}}(D^2)$.\\
We remark (see Lemma V.2 in \cite{GOR}) that if $g\in W^{1,(2,\infty)}_\text{loc}(D^2)$ (meaning that $\nabla g\in L^{2,\infty}_\text{loc}(D^2)$) solves $\operatorname{div}(i\overline{g}\nabla g)=0$ and has a singularity of degree $1$ in $p$, and $\Phi\in W^{1,2}(D^2)$ is a solution of $\operatorname{div}(g\nabla\Phi)=0$, then $\Phi$ has the following form around $p$:
\begin{align}\label{eq: determining-singularity-type}
    \Phi=U\cdot\sum_{j\in \mathbb{N}\smallsetminus\{0\}}a_j\begin{pmatrix}
        \sqrt{j+1}e^{\sqrt{j(j+1)}\,G}g^j\\
        i\sqrt{j}e^{\sqrt{j(j+1)}\,G}g^{-j-1}
    \end{pmatrix}+\Phi_0,
\end{align}
where $U$ is a unitary matrix, $G$ satisfies $-i\overline{g}\nabla g=\nabla^\perp G$, $\Phi_0$ is a constant vector and $\sum_{j\in \mathbb{N}\smallsetminus\{0\}}\lvert a_j\rvert^2<\infty$.\\
We say that $\Phi$ has a \textit{Schoen-Wolfson singularity of type $\Sigma_{j,j+1}$} at $p$ if $a_j\neq 0$ and $a_k=0$ for all $k<j$.

\section{Proof of the theorems}
We first prove Theorem \ref{thm: solving the equation}. Let $p<2$ and $j\in \mathbb{N}\smallsetminus\{0\}$. In Step 1 we will construct an holomorphic map $\phi$ on $D^2$ which is smooth up to the boundary and has infinitely many zeros. In Step 2 we will construct, starting from $\phi$, an $\mathbb{S}^1$-valued map $g\in W^{1,p}\cap W_\text{loc}^{1,(2,\infty)}(D^2)$ satisfying \eqref{eq: thm2-eqs-g}. In Step 4 we will construct (again starting from $\phi$) a map $u\in C^{\sqrt{j^2+j}}(D^2, \mathbb{C})$ whose gradient only vanishes at isolated points in $D^2$, and such that its trace on $\partial D^2$ is smooth. Finally in Step 5 we will show that $u$ and $g$ satisfy \eqref{eq: div-eq-thm}.\\

Let $s\in (0,\frac{2}{p}-1)$.
For any $k\in \mathbb{N}\smallsetminus\{0\}$ set $p_k:=-1+e^{-k}$. Set
\begin{align}
    \phi:=e^{-\frac{1}{(z+1)^s}}\prod_{k\in \mathbb{N}\smallsetminus\{0\}}\left(\frac{z-p_k}{1-p_kz}\right)
\end{align}
in $D^2$, where $z^s$ denotes the branch of the $s$-power with branch cut $(-\infty,0)$ and equal to the real $s$-power on $[0,\infty)$.\\
\textbf{Step 1}:
$\phi$ is a well defined, holomorphic function on $D^2$ satisfying the following properties:
\begin{enumerate}
    \item the zeros of $\phi$ in $D^2$ are precisely the points $p_i$, and all zeros are of order $1$,
    \item $\phi$ extends to a continuous function on $\overline{D}^2$ which is smooth up to the boundary, its trace on $\partial D^2$ has only one zero at $-1$, which has infinite order.
\end{enumerate}
\begin{proof}[Proof of Step 1]
First we show that the Blaschke product
\begin{align}
    \mathscr{P}:=\prod_{k\in \mathbb{N}\smallsetminus\{0\}}\left(\frac{z-p_k}{1-p_kz}\right)
\end{align}
defines a holomorphic function in $D^2$:
for any $z\in D^2$, $k\in \mathbb{N}\smallsetminus\{0\}$
\begin{align}
    1-\frac{z-p_k}{1-p_kz}=e^{-k}\frac{1-z}{1+z-e^{-k}z}.
\end{align}
Thus the series
\begin{align}
    \sum_{k\in \mathbb{N}\smallsetminus\{0\}}\left\lvert 1-\frac{z-p_k}{1-p_kz}\right\rvert
\end{align}
converges uniformly on compact subsets of $D^2$. We deduce that $\mathscr{P}$ defines an holomorphic function on $D^2$ with zeros of order one at the points $p_k$ (see for instance Theorem 15.6 in \cite{Rudin}). Since $e^{-(z+1)^{-s}}$ is holomorphic and different from zero in $D^2$, we conclude that the same properties hold for $\phi$. Notice also that the argument above shows that $\phi$ can be extended to an analytic function in a neighbourhood of any point of $\partial D^2$ but $-1$, so that $\phi$ is smooth up to the boundary away from $-1$, and its extension on $\partial D^2\smallsetminus\{-1\}$ is nowhere zero. On the other hand, since $\lvert \mathscr{P}\rvert\leq 1$ on $D^2$, we have
\begin{align}
    \lvert \phi(z)\rvert\leq \left\lvert \exp\left(-(z+1)^{-s}\right)\right\rvert\leq \exp\left(-\cos\left(s\frac{\pi}{2}\right)\lvert z+1\rvert^{-s}\right),
\end{align}
therefore $\phi$ extends continuously on $\overline{D}^2$ and is equal to zero at $-1$.
We can repeat the same argument for the derivatives of $\phi$: for instance we compute
\begin{align}
    \partial_z\phi=&e^{-(z+1)^{-s}}\sum_{k\in \mathbb{N}\smallsetminus\{0\}}\left(\frac{2e^{-k}-e^{-2k}}{(1-p_kz)^2}\right)\prod_{i\neq k}\left(\frac{z-p_i}{1-p_iz}\right)\\
    \nonumber
    &+s(z+1)^{-s-1}e^{-(z+1)^{-s}}\mathscr{P}.
\end{align}
Notice that for any $n\in \mathbb{N}$, for any $z\in D^2$
there holds
\begin{align}
    \left\lvert \frac{e^{-(z+1)^{-s}}}{(1-p_kz)^n}\right\rvert=&\frac{\lvert e^{-(z+1)^{-s}}\rvert}{\lvert p_k\rvert^n\lvert z+(1-e^{-k})^{-1}\rvert^n}\\
    \nonumber
    \leq& \frac{\exp\left(-\cos\left(s\frac{\pi}{2}\right)\lvert z+1\rvert^{-s}\right)}{\lvert p_k\rvert^n\lvert z+1\rvert^n}\\
    \nonumber
    \leq& C(s,n)\exp\left(-\frac{1}{2}\cos\left(s\frac{\pi}{2}\right)\lvert z+1\rvert^{-s}\right). 
\end{align}
Therefore $\lvert \partial_z \phi\rvert$ tends to zero as $z$ approaches $-1$, so that $\phi$ extends to a $C^1$ function in $\overline{D}^2$. A similar argument can be performed for any derivative of $\phi$, showing that $\phi$ is smooth up to the boundary and that $\phi$ has a zero of infinite order at $-1$.
\end{proof}
Let's define the following functions on $\overline{D}^2$:
\begin{align}
    \rho:=\lvert \phi\rvert;\quad g:=\frac{\phi}{\lvert\phi\rvert}.
\end{align}
\textbf{Step 2}: $g$ belongs to  $W^{1,(2,\infty)}_\text{loc}(D^2)$ and to $W^{1,p}(D^2)$, and satisfies
\begin{align}\label{eq: eqs satisfied by g}
    \operatorname{div}(-i\overline{g}\nabla g)=0\quad\text{and}\quad \operatorname{div}(-i\overline{g}\nabla^\perp g)=-2\pi\sum_{k\in \mathbb{N}}\delta_{p_k}\text{ in }D^2.
\end{align}
\begin{proof}[Proof of Step 2]
    First notice that $\rho$ and $g$ are smooth away from the zeros of $\phi$. Around any zero $p_k$ of $\phi$ we can write $\phi(z)=(z-p_k)h(z)$ for some holomorphic $h$ not vanishing at $p_k$. Thus
    \begin{align}\label{eq: local-description-g}
        g(z)=\frac{z-p_k}{\lvert z-p_k\rvert}\frac{h(x)}{\lvert h(z)\rvert}
    \end{align}
    and
    \begin{align}
        \log\rho(z)=\log\lvert z-p_k\rvert+\log\lvert h(z)\rvert
    \end{align}
    around $p_k$. Since $\frac{h(x)}{\lvert h(z)\rvert}$ and $\log\lvert h(z)\rvert$ are smooth around $p_k$ we conclude that $g$ and $\rho$ belong to $W^{1,(2,\infty)}_\text{loc}(D^2)$.\\
    As $\phi$ is holomorphic, by the Cauchy-Riemann equations there holds
    \begin{align}\label{eq: CR-for-g-rho}
        -i\overline{g}\nabla g=\nabla^\perp\log\rho.
    \end{align}    
    In particular we have
    \begin{align}
        \operatorname{div}(-i\overline{g}\nabla g)=\operatorname{div}(\nabla^\perp\log\rho)=0.
    \end{align}   
    To deduce the second identity in \eqref{eq: eqs satisfied by g}, notice that locally away from the singularities $g$ can be written as $g=e^{i\beta}$ for some harmonic function $\beta$. Thus
    \begin{align}
        \operatorname{div}(-i\overline{g}\nabla^\perp g)=\operatorname{div}(\nabla^\perp\beta)=0
    \end{align}
    away from the zeros of $\phi$.
    Around any zero $p_k$ of $\phi$, writing again $\rho(z)=\lvert z-p_k\rvert\lvert h(z)\rvert$ for $h$ holomorphic with $h(p_k)\neq 0$, we get
    \begin{align}
        \operatorname{div}(-i\overline{g}\nabla^\perp g)(z)=-\operatorname{div}(\nabla\log\rho)(z)=-\Delta \log \lvert z-p_k\rvert-\Delta\log\lvert h(z)\rvert
    \end{align}
    around $p_k$. As $h$ is holomorphic and different from zero, $\log\lvert h\rvert$ is harmonic around $p_k$, therefore
    \begin{align}
         \operatorname{div}(-i\overline{g}\nabla^\perp g)=-\Delta\log\lvert z-p_i\rvert=-2\pi \delta_{p_k}
    \end{align}
    around $p_k$. This completes the proof of \eqref{eq: eqs satisfied by g}.\\
    Next we show that $g\in W^{1,p}(D^2)$: observe that
    \begin{align}
        \left\lvert \nabla\left(\frac{\mathscr{P}}{\lvert\mathscr{P}\rvert}\right)\right\rvert\leq \sum_{k\in \mathbb{N}\smallsetminus\{0\}}\left\lvert\frac{z-p_k}{\lvert z-p_k\rvert^2}-\frac{z-\frac{1}{p_k}}{\lvert z-\frac{1}{p_k}\rvert^2}\right\rvert
    \end{align}
    and that for any $i\in \mathbb{N}\smallsetminus\{0\}$
    \begin{align}\label{eq: estimate-dipole}
        \left\lvert\frac{z-p_k}{\lvert z-p_k\rvert^2}-\frac{z-\frac{1}{p_k}}{\lvert z-\frac{1}{p_k}\rvert^2}\right\rvert\leq 3\frac{\lvert p_k-\frac{1}{p_k}\rvert}{\lvert z-p_k\rvert^2}
    \end{align}
    for any $z\in D^2$.
    A direct computation (considering separately the integral over $B_{2e^{-k}}(-1)\cap D^2$ and the one over $B_{2e^{-k}}(-1)^c\cap D^2$, where we use \eqref{eq: estimate-dipole}) shows that for any $k\in \mathbb{N}\smallsetminus\{0\}$
        \begin{align}
        \int_{D^2} \left\lvert\frac{z-p_k}{\lvert z-p_k\rvert^2}-\frac{z-\frac{1}{p_k}}{\lvert z-\frac{1}{p_k}\rvert^2}\right\rvert^p\leq 16e^{-(2-p)k}.
    \end{align}
    Therefore $\frac{\mathscr{P}}{\lvert \mathscr{P}\rvert}\in W^{1,p}(D^2)$.\\
    On the other hand
    \begin{align}
        \left\lvert\nabla\left( \frac{\exp(-(z+1)^{-s})}{\lvert\exp(-(z+1)^{-s})\rvert }\right)\right\rvert=\left\lvert\nabla\exp(-i\Im(z+1)^{-s})\right\rvert=\lvert\nabla\Im(z+1)^{-s}\rvert\leq s\lvert z+1\rvert^{-1-s},
    \end{align}
    so that
    \begin{align}
        \int_{D^2} \left\lvert\nabla\left(\frac{\exp((-z+1)^{-s})}{\lvert\exp(-(z+1)^{-s})\rvert }\right)\right\rvert^p\leq s^p\int_{B_2(-1)}\lvert z+1\rvert^{-p(1+s)}=s^p\frac{2^{2-p(1+s)}}{2-p(1+s)}
    \end{align}
    (notice that $2-p(1+s)>0$). We conclude that
    \begin{align}
        g=\frac{\mathscr{P}}{\lvert \mathscr{P}\rvert}\frac{\exp(-(z+1)^{-s})}{\lvert\exp(-(z+1)^{-s})\rvert }\in W^{1,p}(D^2).
    \end{align} 
\end{proof}

\textbf{Step 3}: The functions $\rho$ and $\rho^\delta g$ are smooth on $\partial D^2$ for any $\delta>0$.
\begin{proof}[Proof of Step 3]
    Both the angular and radial parts of $e^{(z+1)^{-s}}$ are clearly smooth and different from zero on $\partial D^2$ away from $-1$. Near $-1$ we have
    \begin{align}
        \exp(-(z+1)^{-s})=\exp(-\lvert z+1\rvert^{-s} e^{-i s\varphi}),
    \end{align}
where
\begin{align}
    \varphi=\begin{cases}
        \frac{\theta}{2}&\text{ if }\theta\in (\frac{\pi}{2},\pi)\\
        \frac{\theta}{2}-\pi&\text{ if }\theta\in (\pi,\frac{3}{2}\pi).
    \end{cases}
\end{align}
Since $\lvert \mathscr{P}\rvert=1$ on $\partial D^2$, near $-1$ we have (for $0\leq\theta<\pi$, the other case being analogous)
\begin{align}
    \rho(e^{i\theta})=\exp\left(-(2+2\cos\theta)^{-\frac{s}{2}}\cos\left(\frac{s\theta}{2}\right)\right)
\end{align}
and
\begin{align}
    \rho^\delta g(e^{i\theta})
    =\mathscr{P} \exp\left(i(2+2\cos\theta)^{-\frac{s}{2}}\sin\left(s\frac{\theta}{2}\right)\right)\exp\left(-\delta(2+2\cos\theta)^{-\frac{s}{2}}\cos\left(\frac{s\theta}{2}\right)\right).
\end{align}
Arguing as in Step 1 we see that $\rho\vert_{\partial D^2}$ and $\rho^\delta g\vert_{\partial D^2}$ are smooth also around $-1$.
\end{proof}

\textbf{Step 4}: Set 
\begin{align}\label{eq: expression-u}
    u:=\rho^{\sqrt{j^2+j}-j}\phi^j
\end{align}
Then $u\in C^{\sqrt{j^2+j}}(D^2)$ and away from the points $p_i$, $u$ is smooth and $\nabla u$ vanishes at isolated points.
Moreover $u$ extends to a continuous map on $\overline{D^2}$ and its trace on $\partial D^2$ is smooth.
    \begin{proof}[Proof of Step 4]
    Notice first that the map
    \begin{align}
        \mu_j: \mathbb{R}^2\to\mathbb{R}^2,\quad x\mapsto \lvert x\rvert^{\sqrt{j^2+j}-j}x^j
    \end{align}
    belongs to $C^{\sqrt{j^2+j}}(\mathbb{R}^2)$. Since $\phi$ is smooth up to the boundary, we have $u=\mu_j\circ \phi\in C^{\sqrt{j^2+j}}(D^2)$. As $\mu_j$ is smooth away from zero and $\phi$ is holomorphic, $u$ is smooth away from the points $p_i$.    
    Moreover we notice that away from the points $p_i$, $\nabla u$ vanishes if and only if $\nabla \phi$ vanishes. As $\phi$ is holomorphic and non-constant, its gradient only vanishes at isolated points, therefore the same holds for $u$.\\
    Finally we notice that by Step 3, $u=\rho^{\sqrt{j^2+j}-j}g^j$ is smooth on $\partial D^2$.
    \end{proof}

\textbf{Step 5}: There holds
\begin{align}\label{eq: div(g-nabla-u)}
    \operatorname{div}(g\nabla u)=0\text{ in }D^2.
\end{align}
\begin{proof}[Proof of Step 5]
Let $\varphi\in C^\infty_c(D^2,\mathbb{C})$. As the points $p_k$ are isolated in $D^2$, there are only finitely many of them in $\operatorname{supp}(\varphi)$. Let $I\subset \mathbb{N}$ be the finite index set of the points $p_k$ in $\operatorname{supp}(\varphi)$. As $u\in W^{1,2}(D^2)$ there holds
\begin{align}\label{eq: comp-g-nabla-phi}
    \langle\operatorname{div}(g\nabla u),\varphi\rangle=&-\int_{D^2}g\nabla u\cdot\nabla\varphi=-\lim_{\varepsilon\to 0} \int_{D^2\smallsetminus\bigcup_{k\in I}B_\varepsilon(p_k)}g\nabla u\cdot \nabla\varphi\\
    \nonumber
    =&\lim_{\varepsilon\to 0}\int_{\bigcup_{k\in I}\partial B_\varepsilon(p_k)}g\partial_\nu u\varphi+\lim_{\varepsilon\to 0}\int_{D^2\smallsetminus\bigcup_{k\in I}B_\varepsilon(p_k)}\operatorname{div}(g\nabla u)\varphi.
\end{align}
Here $\nu$ denotes the unit outward pointing vector of $\partial B_{\varepsilon}(p_k)$.
Recall that around any $p_k$ we have $u(z)=(z-p_k)^j\lvert z-p_k\rvert^{\sqrt{j^2+j}-j}w(z)$ for some smooth $w$ with $w(p_k)\neq 0$.
Thus for any $\varepsilon>0$, on $\partial B_\varepsilon(p_k)$ we have
\begin{align}
    \partial_\nu u=\sqrt{j^2+j}\left(\frac{z-p_k}{\lvert z-p_k\rvert}\right)^j\lvert z-p_k\rvert^{\sqrt{j^2+j}-1}w(z)+\left(\frac{z-p_k}{\lvert z-p_k\rvert}\right)^j\lvert z-p_i\rvert^{\sqrt{j^2+j}}\partial_\nu w(z),
\end{align}
so that
\begin{align}
    \lim_{\varepsilon\to 0}\int_{\partial B_\varepsilon(p_i)}g\partial_\nu u\varphi=0
\end{align}
and the first term on the right hand side of \eqref{eq: comp-g-nabla-phi} vanishes. To evaluate the second term on the right hand side of \eqref{eq: comp-g-nabla-phi} we compute, away from the points $p_i$,
\begin{align}\label{eq: computation-div-I}
    \operatorname{div}(g\nabla u)=&\operatorname{div}(g^{j+1}\nabla \rho^{\sqrt{j^2+j}}+j\rho^{\sqrt{j^2+j}}g^j\nabla g)\\
    \nonumber
    =&(j+1)\sqrt{j^2+j}\rho^{\sqrt{j^2+j}-1}g^j\nabla g\cdot\nabla\rho+g^{j+1}\Delta \rho^{\sqrt{j^2+j}}\\
    \nonumber
    &+j \sqrt{j^2+j}\rho^{\sqrt{j^2+j}-1}\nabla\rho\cdot g^j\nabla g+j \rho^{\sqrt{j^2+j}}\operatorname{div}(g^j\nabla g).
\end{align}
Notice that since $\nabla g\cdot\nabla\rho=0$, the first and the third terms vanish.
We have
\begin{align}\label{eq: computation-div-II}
    g^{j+1}\Delta \rho^{\sqrt{j^2+j}}=&g^{j+1}\operatorname{div}(\sqrt{j^2+j}\rho^{\sqrt{j^2+j}}\nabla\log\rho)\\
    \nonumber =&(j^2+j)g^{j+1}\rho^{\sqrt{j^2+j}}\lvert\nabla \log\rho\rvert^2+\sqrt{j^2+j}g^{j+1}\rho^{\sqrt{j^2+j}}\Delta\log\rho.
\end{align}
The second term vanishes because $\log \rho$ is harmonic away from the points $p_k$ (since $\phi$ is holomorphic).
On the other hand, since $\operatorname{div}(\overline{g}\nabla g)=0$ we have
\begin{align}\label{eq: computation-div-III}
    j\rho^{\sqrt{j^2+j}} \operatorname{div}(g^j\nabla g)=&j\rho^{\sqrt{j^2+j}}\operatorname{div}(g^{j+1}\overline{g}\nabla g)\\
    \nonumber
    =&-j(j+1)\rho^{\sqrt{j^2+j}}g^{j+1}(i\overline{g}\nabla g)\cdot (i\overline{g}\nabla g)\\
    \nonumber
    =&-(j^2+j)\rho^{\sqrt{j^2+j}}g^{j+1}\lvert\nabla\log\rho\rvert,
 \end{align}  
where we used that by \eqref{eq: CR-for-g-rho} $\lvert i\overline{g}\nabla g\rvert=\lvert\nabla\log\rho\rvert$.
Combining \eqref{eq: computation-div-I}, \eqref{eq: computation-div-II} and \eqref{eq: computation-div-III} we
see that $\operatorname{div}(g\nabla u)=0$ away from the points $p_k$. Thus
\eqref{eq: comp-g-nabla-phi} implies that $\langle\operatorname{div}(g\nabla u),\varphi\rangle=0$. As this holds for any $\varphi\in C^\infty_c(D^2)$, the claim follows.
\end{proof}
This concludes the proof of Theorem \ref{thm: solving the equation}.
Next we apply Theorem \ref{thm: solving the equation} to prove Theorem \ref{thm: main theorem}.
\begin{proof}[Proof of Theorem \ref{thm: main theorem}] Let $p<2$ and $j\in \mathbb{N}\smallsetminus\{0\}$ and let $g$ and $u$ be as in Theorem \ref{thm: solving the equation}. Since $\operatorname{div}(g\nabla u)=0$, there exists $v\in W^{1,\infty}(D^2, \mathbb{C})$ such that $\nabla^\perp v=g\nabla u$.
Set
\begin{align}
    \Phi=\begin{pmatrix}
    u\\-\overline{v}
    \end{pmatrix}.
\end{align}
    Then $\Phi$ belongs to $W^{1,\infty}(D^2,\mathbb{C}^2)$ and it is smooth away from the points $p_k$.
    Recall that we are identifying $\mathbb{C}^2$ with the algebra of quaternions with basis elements $1,I,J,K$. 
    Then the identity $\nabla^\perp v=g\nabla u$ implies that
    \begin{align}\label{eq: derivatives-relation}
    \partial_x\Phi=\overline{g}J\partial_y\Phi.
    \end{align}
    We deduce that $\Phi$ is weakly conformal (as $\langle \partial_x\Phi, \partial_y\Phi\rangle=0$ and $\lvert\partial_x\Phi\rvert^2=\lvert\partial_y\Phi\rvert^2$ a.e.) and Lagrangian (as $\langle \partial_x\Phi,I\partial_y\Phi\rangle=0$ a.e.).\\
    We claim that the Lagrangian angle of $\Phi$ is given by $\overline{g}$. With the help of \eqref{eq: Lagrangian angle} we compute (for any point where $\lvert \nabla \Phi\rvert\neq 0)$
    \begin{align}
        e^{-2\lambda}\det[\partial_x\Phi,\partial_y\Phi]=&e^{-2\lambda}(-\partial_xu\partial_y\overline{v}+\partial_x\overline{v}\partial_y u)\\
        \nonumber
        =&e^{-2\lambda}\left(\partial_x u\overline{g}\partial_x\overline{u}+\overline{g}\partial_y \overline{u}\partial_y u\right)=\overline{g},
    \end{align}
    where $e^{2\lambda}=\lvert\partial_x\Phi\rvert^2=\lvert\partial_y\Phi\rvert^2=2\lvert\partial_x u\rvert^2=2\lvert\partial_y u\rvert^2$.
    Since $\overline{g}$ satisfies $\operatorname{div}(i\overline{g}\nabla g)=0$,
    we conclude that $\Phi$ is Hamiltonian stationary.\\
    As the gradient of the map $u$ only vanishes at isolated points, the same holds for $\Phi$. Since the Lagrangian angle is continuous away from the points $p_k$, we conclude that $\Phi$ is a branched immersion away from the points $p_k$.
    Next we observe that the map $v$, determined up to an additive constant, can be chosen to be
    \begin{align}\label{eq: expression-v}
        v=i\frac{\sqrt{j}}{\sqrt{j+1}}\rho^{\sqrt{j^2+j}}g^{j+1}=i\frac{\sqrt{j}}{\sqrt{j+1}}\nu_j\circ\phi,
    \end{align}
    where
        \begin{align}
        \nu_j: \mathbb{R}^2\to\mathbb{R}^2,\quad x\mapsto \lvert x\rvert^{\sqrt{j^2+j}-(j+1)}x^{j+1}.
    \end{align}
    Observe that $\nu_j\in C^{\sqrt{j^2+j}}(\mathbb{R}^2)$. As $\phi$ is smooth up to the boundary, $v\in C^{\sqrt{j^2+j}}(D^2)$. Since $u$ also lies in $C^{\sqrt{j^2+j}}(D^2)$ we conclude that $\Phi\in C^{\sqrt{j^2+j}}(D^2)$.
    Moreover Step 4 in the proof of Theorem \ref{thm: solving the equation} implies that $\rho^\frac{\sqrt{j^2+j}}{j}g$ is smooth on $\partial D^2$, therefore also $v=i\sqrt{\frac{j}{j+1}}(\rho^\frac{\sqrt{j^2+j}}{j}g)^j$ is smooth on $\partial D^2$. Since also $u$ has a smooth trace on $\partial D^2$, we conclude that $\Phi$ has a smooth trace on $\partial D^2$.\\
    Finally, comparing the explicit expression for $\Phi$ (given by \eqref{eq: expression-u} and \eqref{eq: expression-v}) with \eqref{eq: determining-singularity-type}, we see that at any point $p_k$, $\Phi$ has a singularity of type $\Sigma_{j,j+1}$.
\end{proof}

\section{Related questions}
In this section we discuss condition \eqref{eq: Neumann-bdy-cond} and how it relates to our construction.\\
First we recall that if a map $g\in W^{1,p}(D^2,\mathbb{S}^1)$ for some $p>1$ satisfies
\begin{align}
    \operatorname{div}(i\overline{g}\nabla g)=0,
\end{align}
then there exists $G\in W^{1,p}(D^2,\mathbb{R})$ such that
\begin{align}\label{eq: property-G}
    -i\overline{g}\nabla g=\nabla^\perp G.
\end{align}
The condition
$i\overline{g}\partial_\nu g=0$ on $\partial D^2$ can be given the following meaning: we require the map $G$ to have constant trace on $\partial D^2$.\\
As observed in \cite{Schoen-SLS}, this condition arises naturally in Lagrangian free-boundary problems where the constraint manifold is a complex surface (see also Remark IV.4 in \cite{GOR}).
Originally we were hoping to obtain a map as in Theorem \ref{thm: main theorem} satisfying in addition $\overline{g}\partial_\nu g=0$ on $\partial D^2$ in the sense specified above.
Unfortunately our method does not allow to obtain such an example.
The proof of Theorem \ref{thm: solving the equation} in Section 3, in fact, is based on the construction of a non-zero holomorphic function $\phi$ on $D^2$ with infinitely many zeros and finite Dirichlet energy. The $\mathbb{S}^1$-valued map $g$ is then obtained as $g=\frac{\phi}{\lvert \phi\rvert}$, so that the function $G$ appearing in \eqref{eq: property-G} can be chosen to be $\log\lvert \phi\rvert$. To get an example satisfying $\overline{g}\partial_\nu g=0$,
we would need to have $\lvert \phi\rvert$ constant on $\partial D^2$, but this is not possible, because of the next lemma.
\begin{lem}\label{lem: no-desired-hol-fct}
    If an holomorphic map $\phi$ on $D^2$ satisfies the following properties
\begin{enumerate}
    \item $\phi$ has infinitely many zeros,
    \item $\displaystyle\int_{D^2}\lvert\nabla \phi\rvert^2<\infty$,
    \item $\lvert\phi\rvert$ is constant on $\partial D^2$,
\end{enumerate}
then $\phi\equiv 0$.
\end{lem}
Lemma \ref{lem: no-desired-hol-fct} is a consequence of the following result, based on \cite{SU-bdry}, Section 4, and whose idea has been widely used in \cite{BN-I} and \cite{BN-II}.
\begin{lem}\label{lem: Brezis-Nirenberg}
Let $\psi\in H^\frac{1}{2}(\partial D^2,\mathbb{S}^1)$ and let $\phi$ be its harmonic
extension in $D^2$. Then $\lvert \phi(z)\rvert\to 1$ uniformly as $\lvert z\rvert\to 1$.    
\end{lem}
\begin{proof}
Since ${\phi}$ is harmonic, $\phi(r,\theta)=P_r\ast \psi(\theta)$, where for any $r\in (0,1)$ the Poisson kernel $P_r$ is given by
\begin{align}
    P_r(\theta)=\frac{1-r^2}{2\lvert re^{i\theta}-1\rvert^2}.
\end{align}
Let $\theta\in [0,2\pi)$, $r\in (0,1)$. We estimate (exploiting the fact that $\lvert\phi\rvert=1$ on $\partial D^2$)
\begin{align}\label{eq: dist-estimate}
    \operatorname{dist}^2(\phi(r,\theta),\mathbb{S}^1)\leq& \fint_{B_{1-r}(\theta)}\lvert P_r\ast\psi(\theta)-\psi(\sigma)\rvert^2d\sigma\\
    \nonumber
    =&\fint_{B_{1-r}(\theta)}\left\lvert \int_0^{2\pi}P_r(\alpha-\theta)(\psi(\alpha)-\psi(\sigma))d\alpha\right\rvert^2 d\sigma\\
    \nonumber
    \leq &\fint_{B_{1-r}(\theta)}\int_0^{2\pi} P_r(\alpha-\theta)\lvert\psi(\alpha)-\psi(\sigma)\rvert^2d\alpha d\sigma,
\end{align}
where $B_{1-r}(\theta)$ denotes a ball in $\mathbb{R}$.\\
Notice that $1-r<\lvert re^{i\varphi}-1\rvert$ for any $\varphi$, therefore $\lvert e^{i\alpha}-e^{i\sigma}\rvert^2<9\lvert re^{i(\alpha-\sigma)}-1\rvert^2$ if $\sigma\in B_{1-r}(\theta)$, $\alpha\in B_{2(1-r)}(\theta)$. Thus
\begin{align}\label{eq: first integral}
    &\fint_{B_{1-r}(\theta)}\int_{B_{2(1-r)}(\theta)}\lvert P_r(\alpha-\theta)(\psi(\alpha)-\psi(\sigma)\rvert^2d\alpha d\sigma\\
    \nonumber
    \leq&
    \frac{9}{2}(1-r^2)\fint_{B_{1-r}(\theta)}\int_{B_{2(1-r)}(\theta)}\frac{\lvert\psi(\alpha)-\psi(\sigma)\rvert^2}{\lvert e^{i\alpha}-e^{i\sigma}\rvert^2}d\alpha d\sigma\\
    \nonumber
    =&\frac{9}{4}(1+r)\int_{B_{1-r}(\theta)}\int_{B_{2(1-r)}(\theta)}\frac{\lvert\psi(\alpha)-\psi(\sigma)\rvert^2}{\lvert e^{i\alpha}-e^{i\sigma}\rvert^2}d\alpha d\sigma.
\end{align}
On the other hand
there holds $r\lvert e^{i\varphi}-1\rvert\leq \lvert re^{i\varphi}-1\rvert$ for any $\varphi$, therefore
\begin{align}\label{eq: second integral}
    &\fint_{B_{1-r}(\theta)}\int_{B_{2(1-r)}^c(\theta)} P_r(\alpha-\theta)\lvert\psi(\alpha)-\psi(\sigma)\rvert^2d\alpha d\sigma\\
    \nonumber
    \leq&
    \frac{1-r^2}{8r^2(1-r)}\int_{B_{1-r}(\theta)}\int_{B_{2(1-r)}^c(\theta)}\frac{\lvert \psi(\alpha)-\psi(\sigma)\rvert^2}{\lvert e^{i(\alpha-\theta)}-1\rvert^2}d\sigma d\alpha\\
    \nonumber
    \leq& \frac{1+r}{8r^2}\int_{B_{1-r}(\theta)}\int_{B_{2(1-r)}^c(\theta)}\frac{\lvert \psi(\alpha)-\psi(\sigma)\rvert^2}{\lvert e^{i\alpha}-e^{i\sigma}\rvert^2}d\alpha d\sigma,
\end{align}
where $B_{2(1-r)}^c(\theta)$ denotes the complement of $B_{2(1-r)}(\theta)$ in $\quot{\mathbb{R}}{2\pi\mathbb{Z}}$, and in the last step we used the fact that if $\alpha\in B_{2(1-r)}^c(\theta)$ and $\sigma\in B_{1-r}(\theta)$, then
\begin{align}
    \lvert e^{i\alpha}-e^{i\sigma}\rvert\leq \lvert\alpha-\sigma\rvert\leq \frac{3}{2}\lvert\alpha-\theta\rvert\leq \frac{3\pi}{4}\lvert e^{i(\alpha-\theta)}-1\rvert.
\end{align}
Notice that the expressions in \eqref{eq: first integral} and \eqref{eq: second integral} tend to zero uniformly as $r$ approaches 1, in fact $\frac{\lvert\psi(\alpha)-\psi(\sigma)\rvert^2}{\lvert e^{i\alpha}-e^{i\sigma}\rvert^2}$ lies in $L^1(\partial D^2\times\partial D^2)$ and thus is uniformly integrable, as $\psi\in H^\frac{1}{2}(\partial D^2)$. Therefore \eqref{eq: dist-estimate} implies that $\operatorname{dist}(\phi(r, \theta), \mathbb{S}^1)$ tends to zero uniformly (as a function of $\theta$) as $r$ tends to $1$.
\end{proof}
\begin{proof}[Proof of Lemma \ref{lem: no-desired-hol-fct}]
Assume by contradiction that $\phi$ is a non-zero holomorphic map on $D^2$ satisfying (1), (2) and (3).
Multiplying $\phi$ by a constant we may assume that $\lvert\phi\rvert\equiv 1$ on $\partial D^2$. Since $\phi\in W^{1,2}(D^2)$, its trace on $\partial D^2$ lies in $H^\frac{1}{2}(\partial D^2)$ and therefore defines a map in $H^\frac{1}{2}(\partial D^2,\mathbb{S}^1)$ as in Lemma \ref{lem: Brezis-Nirenberg}. We deduce that there exists some $r\in (0,1)$ such that $\lvert \phi\rvert>\frac{1}{2}$ in $D^2\smallsetminus B_r(0)$, thus all the zeros of $\phi$ must lie in $B_r(0)$. But since $\phi$ is holomorphic in $D^2$, its zeros are isolated (they can only accumulate at the boundary of $D^2$), therefore $\phi$ can only have finitely many zeros in $\overline{B_r(0)}$. This contradicts property (1).
\end{proof}

\printbibliography

\end{document}